\documentclass{amsart}

\usepackage{amsmath,amsthm,amssymb}
\usepackage{tikz,graphicx}
\usepackage{standalone}
\usepackage{caption}
\usepackage{enumitem}
\usepackage{hyperref}

\usetikzlibrary{intersections,calc}

\newcommand{\makepoints}{ 
	\foreach \n in {1,...,4} { \coordinate (\n) at (\n*90-135:0.5cm); }; }
\newcommand{\drawpoints}{ 
	\foreach \n in {1,...,4} { \draw (\n) node [littledot] {}; }; }

\definecolor{blue-violet}{rgb}{0.54, 0.17, 0.89}

\tikzstyle{PurpleLine}=[line width=0.3mm,color=blue-violet,text=black]
\tikzstyle{PurplePoly}=[PurpleLine,fill=blue-violet!30]
\tikzstyle{BlueLine}=[line width=0.3mm,color=blue,text=black]
\tikzstyle{BluePoly}=[BlueLine,fill=blue!20]
\tikzstyle{RedLine}=[line width=0.3mm,color=red,text=black]
\tikzstyle{RedPoly}=[RedLine,fill=red!20]
\tikzstyle{GreenLine}=[thick,draw=black!30!green,text=black]
\tikzstyle{GreenPoly}=[thick,draw=green!50!black,fill=green!30,join=bevel]
\tikzstyle{OrangeLine}=[thick,color=orange]
\tikzstyle{GrayLine}=[thick,color=black!50!gray]
\tikzstyle{GrayPoly}=[GrayLine,fill=gray!20]
\tikzstyle{dot}=[shape=circle,draw,color=black,fill=black,inner sep=1.5pt]
\tikzstyle{bigdot}=[dot,inner sep=2pt]
\tikzstyle{littledot}=[dot,inner sep=1.2pt]
\tikzstyle{tinydot}=[dot,inner sep=0.6pt]
\tikzstyle{disk}=[thick,shape=circle,draw,color=black,fill=yellow!10]
\tikzstyle{plate}=[thick,shape=rectangle,draw,color=black,fill=yellow!10,
	rounded corners,minimum size=1.1cm]

\tikzstyle{dot}=[shape=circle,draw,color=black,fill=black,inner sep=1.5pt]
\tikzstyle{opendot}=[dot,fill=white]

\newcommand{\makelinepoints}{
    \foreach \n in {1,2,3,4} {
        \coordinate (\n) at (-1.25+\n*0.5,0);
    }
}
\newcommand{\drawlinepoints}{
    \foreach \n in {1,2,3,4} {
        \draw (\n) node [littledot] {};
    }
}

\newcommand{\maketpoints}{
	\coordinate (0) at (0,0.3);
	\foreach \n in {1,2,3,4} {
        \coordinate (\n) at (-0.5+\n*0.2,-0.2);
    }
}
\newcommand{\drawtpoints}{
    \foreach \n in {0,1,2,3,4} {
        \draw (\n) node [tinydot] {};
    }
}

\tikzstyle{YellowRect} = [shape=rectangle,rounded corners,draw,fill=yellow!40,minimum width = 2cm, minimum height=1cm]

\theoremstyle{plain}
\newtheorem{thm}{Theorem}[section]
\newtheorem{lem}[thm]{Lemma}

\newtheorem{prop}[thm]{Proposition}

\newtheorem{question}[thm]{Question}

\newtheorem{mainthm}{Theorem}

\theoremstyle{definition}
\newtheorem{defn}[thm]{Definition}
\newtheorem{rem}[thm]{Remark}

\newtheorem{example}[thm]{Example}

\newcommand{\C}{\mathbb{C}}

\newcommand{\NC}{\textsc{NC}}

\newcommand{\conv}{\textsc{Conv}}
\newcommand{\bool}{\textsc{Bool}}

\begin{document}

\title{Noncrossing Partitions from Cones and Semicircles}

\author[Dougherty]{Michael Dougherty}
\email{doughemj@lafayette.edu}
\address{Dept. of Mathematical Sciences, Lafayette College,
  Easton, PA 18042}
\author[Fang]{Kaiyi Fang}
\email{fangk@lafayette.edu}
\address{Dept. of Mathematical Sciences, Lafayette College,
  Easton, PA 18042}
\author[Jiang]{Yunting Jiang}
\email{jiangy@lafayette.edu}
\address{Dept. of Mathematical Sciences, Lafayette College,
  Easton, PA 18042}
\author[Lin]{Edgar Lin}
\email{linb@lafayette.edu}
\address{Dept. of Mathematical Sciences, Lafayette College,
  Easton, PA 18042}
\author[Lindenmuth]{Lucas Lindenmuth}
\email{lindenml@lafayette.edu}
\address{Dept. of Mathematical Sciences, Lafayette College,
  Easton, PA 18042}
\author[Pokras]{Eleanor Pokras}
\email{pokrasm@lafayette.edu}
\address{Dept. of Mathematical Sciences, Lafayette College,
  Easton, PA 18042}
\author[Root]{Gina Root}
\email{roota@lafayette.edu}
\address{Dept. of Mathematical Sciences, Lafayette College,
  Easton, PA 18042}

\date{\today}

\begin{abstract}
    For each finite configuration of distinct points in the 
    plane, there is an associated lattice of noncrossing partitions.
    When these points form the vertices of a convex polygon, the
    result is the classical noncrossing partition lattice, 
    which is enumerated by the Catalan numbers and satisfies many
    other useful properties. In this article, we examine three 
    variations of this lattice which arise when the starting 
    configuration is allowed to have points on the sides
    of a convex polygon rather than just the vertex set.
\end{abstract}

\maketitle

\section{Introduction}

Let $P$ be a finite set of points in the plane. A partition of $P$ into blocks 
is \emph{noncrossing} if the convex hulls of the blocks are pairwise disjoint,
and the set of all such partitions---denoted $\NC(P)$---is partially ordered by refinement,
forming a subposet of the full lattice of partitions of $P$. When $P$ is the
vertex set of a convex $n$-gon, the poset $\NC(P)$ is the classical 
\emph{noncrossing partition lattice} $\NC(n)$, which was introduced by Kreweras
\cite{kreweras72} and has since taken on an important role in combinatorics,
geometric group theory, and other fields. See \cite{baumeister19} and \cite{mccammond06}
for surveys. 
At another extreme, if $P$ is a configuration of $n$ points which are all collinear,
then $\NC(P)$ is isomorphic to the Boolean lattice of height $n-1$, i.e. the poset of
all subsets of $\{1,\ldots,n-1\}$ under inclusion. 

The Boolean lattice $\bool(n)$ and the classical lattice of noncrossing partitions 
$\NC(n)$ satisfy a number of interesting poset properties. To name only a few, both
of these lattices are bounded and graded, both are self-dual and admit 
symmetric chain decompositions \cite{de-bruijn,simion-ullman91}, and they are enumerated by 
interesting integer sequences (powers of 2 and the \emph{Catalan numbers}, respectively).
Both lattices are also of interest in the study of Artin groups due to their relationship 
to Garside structures for free abelian groups and braid groups, respectively
\cite{bessis03,brady-watt02}.

The appearance of these two important 
posets from the same construction suggests studying a natural generalization which
incorporates both examples.

\begin{question}\label{ques:nc-convex-boundary}
	If $P$ is a finite set of points which lie in the boundary of a convex
	polygon (rather than merely its vertices), then what can be said about $\NC(P)$?
\end{question}

It is straightforward to see that $\NC(P)$ is a bounded lattice for any choice of
configuration $P$ \cite[Proposition~2.3]{cdhm24}, but other properties vary. 
For example, there are large classes of configurations $P$ for which $\NC(P)$ 
is also counted by the Catalan numbers (without 
being isomorphic to the classical lattice of noncrossing partitions), yet there are
other configurations for which the corresponding lattice of noncrossing partitions
is not rank-symmetric (or even graded)---see Figure~\ref{fig:ungraded-configurations}. 
Our first theorem demonstrates that the configurations considered in 
Question~\ref{ques:nc-convex-boundary} always produce graded lattices.

\begin{figure}
	\centering
	\begin{tikzpicture}
	    \begin{scope}[shift={(-4,0)}]
	        \node [shape=rectangle,draw,fill=green!20, inner sep = 1.8cm] () at (0,0) {};
	    
			\node[dot] () at (90:1.4cm) {};
			\node[dot] () at (210:1.4cm) {};
			\node[dot] () at (330:1.4cm) {};
			\node[dot] () at (150:1.4cm) {};
			\node[dot] () at (270:1.4cm) {};
			\node[dot] () at (30:1.4cm) {};
	
	    \end{scope}
	
	    \begin{scope}[shift={(0,0)}]
	        \node [shape=rectangle,draw,fill=green!20, inner sep = 1.8cm] () at (0,0) {};
	    
			\node[dot] () at (90:1.4cm) {};
			\node[dot] () at (210:1.4cm) {};
			\node[dot] () at (330:1.4cm) {};
			\node[dot] () at (150:0.7cm) {};
			\node[dot] () at (270:0.7cm) {};
			\node[dot] () at (30:0.7cm) {};
	        
	    \end{scope}
	
	    \begin{scope}[shift={(4,0)}]
	        \node [shape=rectangle,draw,fill=green!20, inner sep = 1.8cm] () at (0,0) {};
	    
			\node[dot] () at (90:1.4cm) {};
			\node[dot] () at (210:1.4cm) {};
			\node[dot] () at (330:1.4cm) {};
			\node[dot] () at (150:0.3cm) {};
			\node[dot] () at (270:0.3cm) {};
			\node[dot] () at (30:0.3cm) {};
	        
	    \end{scope}
	\end{tikzpicture}
	\caption{The lattice of noncrossing partitions for the leftmost
	configuration is both graded and rank-symmetric, whereas the lattice
	of noncrossing partitions for the middle configuration is graded but not
	rank-symmetric, and the lattice of noncrossing partitions for the rightmost
	configuration is not even graded.}
	\label{fig:ungraded-configurations}
\end{figure}
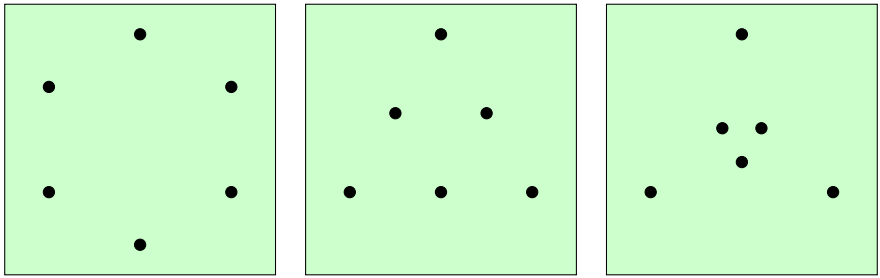

\begin{mainthm}[Theorem~\ref{thm:graded}]\label{mainthm:graded}
	If $P$ is a finite subset of $\mathbb{C}$ which lies on the boundary of a convex polygon,
	then $\NC(P)$ is graded.
\end{mainthm}

For other poset properties, we investigate the noncrossing partitions for 
three special types of configurations which arise in the setting of
Question~\ref{ques:nc-convex-boundary}. 


\begin{defn}\label{def:three-configurations} 
For each pair of nonnegative integers $m$ and $n$, we define three
different types of configurations in the plane, illustrated in Figure~\ref{fig:three-configurations}.
	\begin{enumerate}
		\item An \emph{open cone configuration} $U_{m,n}$ consists of $m+n$ points on the boundary 
		of an affine convex cone with $m$ points on one bounding ray, $n$ points on the other ray,
		and no point on their intersection.
		
		\item A \emph{closed cone configuration} $V_{m,n}$ consists of $m+n+1$ points on the boundary 
		of an affine convex cone with $m$ points on one bounding ray, $n$ points on the other ray,
		and one point on their intersection. Removing the intersection point turns $V_{m,n}$ into $U_{m,n}$.

		\item A \emph{semicircular configuration} $S_{m,n}$ consists of $m+n+2$ points on a
		semicircle, with $m+2$ points on the flat side (including the corners) and $n$ points on 
		the round side (not including the corners).
	\end{enumerate}
	Note that the isomorphism type of the noncrossing partition lattice for each
	type of configuration is determined only by $m$ and $n$, so it is well-defined
	to write e.g. $\NC(U_{m,n})$ without specifying the coordinates of the configuration.
	Also, note that if $P$ consists of $m$ points on one line and $n$ points on a distinct 
	parallel line, then $\NC(P)$ is isomorphic to $\NC(U_{m,n})$. 
\end{defn}

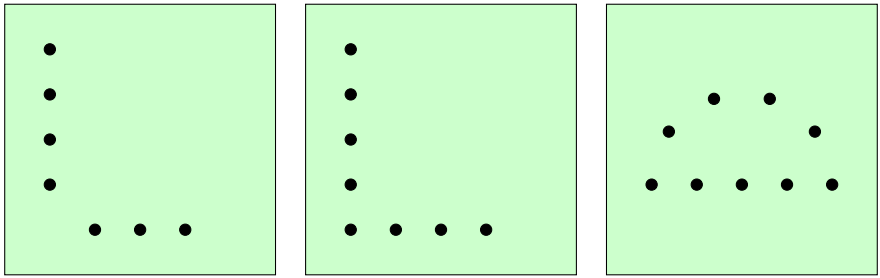
\begin{figure}
	\centering
	\begin{tikzpicture}
	    \begin{scope}[shift={(-4,0)}]
	        \node [shape=rectangle,draw,fill=green!20, inner sep = 1.8cm] () at (0,0) {};
	    
			\node[dot] () at (-1.2,1.2) {};
			\node[dot] () at (-1.2,0.6) {};
			\node[dot] () at (-1.2,0) {};
			\node[dot] () at (-1.2,-0.6) {};
			\node[dot] () at (-0.6,-1.2) {};
			\node[dot] () at (0,-1.2) {};
			\node[dot] () at (0.6,-1.2) {};
	
	    \end{scope}
	
	    \begin{scope}[shift={(0,0)}]
	        \node [shape=rectangle,draw,fill=green!20, inner sep = 1.8cm] () at (0,0) {};
	    
			\node[dot] () at (-1.2,1.2) {};
			\node[dot] () at (-1.2,0.6) {};
			\node[dot] () at (-1.2,0) {};
			\node[dot] () at (-1.2,-0.6) {};
			\node[dot] () at (-0.6,-1.2) {};
			\node[dot] () at (0,-1.2) {};
			\node[dot] () at (0.6,-1.2) {};
			\node[dot] () at (-1.2,-1.2) {};
	        
	    \end{scope}
	
	    \begin{scope}[shift={(4,0)}]
	        \node [shape=rectangle,draw,fill=green!20, inner sep = 1.8cm] () at (0,0) {};
	    
			\node[dot] () at (0,-0.6) {};
			\node[dot] () at (-0.6,-0.6) {};
			\node[dot] () at (-1.2,-0.6) {};
			\node[dot] () at (0.6,-0.6) {};
			\node[dot] () at (1.2,-0.6) {};
			\node[dot] () at ($(36:1.2)+(0,-0.6)$) {};
			\node[dot] () at ($(72:1.2)+(0,-0.6)$) {};
			\node[dot] () at ($(108:1.2)+(0,-0.6)$) {};
			\node[dot] () at ($(144:1.2)+(0,-0.6)$) {};
	        
	    \end{scope}
	\end{tikzpicture}
	\caption{From left to right: the open cone configuration $U_{3,4}$, 
	the closed cone configuration $V_{3,4}$, and the semicircular configuration
	$S_{3,4}$.}
	\label{fig:three-configurations}
\end{figure}

Our next main theorem addresses symmetric chain decompositions in 
noncrossing partition lattices for each of the three configuration types above. 

\begin{mainthm}[Theorems~\ref{thm:uv-scd} and \ref{thm:s-scd}]\label{mainthm:scd}
	For all $m$ and $n$, the bounded graded lattices 
	$\NC(U_{m,n})$, $\NC(V_{m,n})$ and $\NC(S_{m,n})$ admit symmetric chain decompositions.
	Consequently, all three lattices are rank-symmetric.
\end{mainthm}

It is worth noting that despite being rank-symmetric, these posets are generally
not self-dual. Combined with the observations presented in \cite{cdhm24}, it would
seem that self-duality is a relatively rare property among noncrossing partition lattices
from configurations.

Our final main theorem provides multivariate generating functions which enumerate each of the
three lattices. 

\begin{mainthm}[Theorems~\ref{thm:u-gf}, \ref{thm:v-gf} and \ref{thm:s-gf}]\label{mainthm:gf}
	Let $u_{m,n}$, $v_{m,n}$ and $s_{m,n}$ denote the sizes of $\NC(U_{m,n})$, $\NC(V_{m,n})$
	and $\NC(S_{m,n})$ respectively. Then the corresponding multivariate generating functions
	are as follows:
	\begin{enumerate}
		\item $\displaystyle U(x,y) = \sum_{\substack{m,n\geq 0}} u_{m,n}x^m y^n = \frac{x + y - 2xy}{1-2x-2y+3xy}$;
		\item $\displaystyle V(x,y) = \sum_{\substack{m,n\geq 0}} v_{m,n}x^m y^n = \frac{1}{1-2x-2y+3xy}$;
		\item $\displaystyle S(x,y) = \sum_{\substack{m,n\geq 0}} s_{m,n}x^m y^n = \left(\frac{C(y)-1-y}{y^2}\right)
		\left(\frac{1}{1-x(1+C(y))}\right)$,
	\end{enumerate}
	where $C(y) = (1-\sqrt{1-4y})/2y$ is the generating function for the Catalan numbers.
\end{mainthm}

It is worth noting that, in a sense, the generating function for counting semicircular 
noncrossing partitions interpolates between $C(y)$, the generating function which
enumerates classical noncrossing partitions, and $1/(1-2x)$, the generating function
which enumerates noncrossing partitions of collinear points. 

The article is structured as follows. In Section~\ref{sec:prelim}, we introduce some preliminary material on 
posets and noncrossing partitions before proving Theorem~\ref{mainthm:graded} and investigating a particular
example which is common to all three configuration types. The proofs of Theorems~\ref{mainthm:scd} and \ref{mainthm:gf}
are divided across Sections~\ref{sec:cone-config} and \ref{sec:semicirc-config}, which examine 
(open and closed) cone configurations and semicircular configurations, respectively.

\section{Noncrossing partitions from configurations}\label{sec:prelim}

In this section we recall some background on posets and noncrossing partitions. See
\cite[Section 2]{cdhm24} for more detail. 

\begin{defn}\label{def:partition-lattice}
	Let $S$ be a finite set. A \emph{partition} of $S$ is a collection of pairwise disjoint
	subsets of $S$ (called \emph{blocks}) whose union is $S$. The set of all partitions of $S$, denoted 
	$\Pi(S)$, is  partially ordered under refinement, i.e. $\pi \leq \mu$ in
	$\Pi(S)$ if every block of $\mu$ is a union of blocks in $\pi$. Moreover, 
	$\Pi(S)$ has the following properties.
	\begin{itemize}
		\item $\Pi(S)$ is \emph{bounded}: it has a unique minimum $\hat{0}$ and a
		unique maximum $\hat{1}$.	
		\item $\Pi(S)$ is a \emph{lattice}: each pair of elements $\pi$ and $\mu$ has 
		a unique meet (greatest lower bound) $\pi \wedge \mu$ and a unique join
		(least upper bound) $\pi \vee \mu$.
		\item $\Pi(S)$ is \emph{graded}: if $bl(\pi)$ is the number of blocks in $\pi$, then
		the function $\rho\colon \Pi(S) \to \mathbb{N}$ given by
		$\rho(\pi) = |S| - bl(\pi)$ is a \emph{rank function} for $\Pi(S)$, which
		means that $\rho(\pi) \leq \rho(\mu)$ when $\pi \leq \mu$ and $\rho(\pi) + 1 = \rho(\mu)$
		when $\pi < \mu$ and there is no $\eta \in \Pi(S)$ with $\pi < \eta < \mu$.
	\end{itemize}
	If $|S| = n$, then $\hat{0}$ has rank $0$ and $\hat{1}$ has rank $n-1$. The \emph{atoms}
	and \emph{coatoms} of $\Pi(S)$ are the elements of rank $1$ and $n-2$, respectively.
\end{defn}

When the set $S$ has geometric meaning, we can restrict $\Pi(S)$ to an interesting subposet.

\begin{defn}
	Let $P$ be a finite subset of $\C$ (referred to here as a \emph{configuration}). For each subset $A\subseteq P$, its \emph{convex hull} 
	$\conv(A)$ is the smallest convex set in $\C$ which contains $P$. Since $P$ is assumed
	to be finite, $\conv(A)$ is necessarily a convex polygon with up to $|A|$ vertices.
	A partition of the configuration $P$ is \emph{noncrossing} if the convex hulls of its
	blocks are pairwise disjoint, and the set of all noncrossing partitions $\NC(P)$ is
	a subposet of the partition lattice $\Pi(P)$.
\end{defn}

The properties held by the poset $\NC(P)$ can vary dramatically based on the choice of $P$.
For example, $\NC(P)$ is always a bounded lattice, but it is not always graded \cite{cdhm24}. 
In the cases presented in this article, however, we do not need to worry about these exceptions.

\begin{thm}\label{thm:graded}
	If $P$ is a finite subset of $\C$ which lies on the boundary of a convex polygon,
	then $\NC(P)$ is graded.
\end{thm}

\begin{proof}
	Let $\rho$ be the rank function for $\Pi(P)$ described in Definition~\ref{def:partition-lattice}
	and let $\pi < \mu$ be a covering relation in $\NC(P)$, which is to say that $\pi$ and $\mu$ 
	are noncrossing partitions of $P$ with the property that there is no $\eta \in \NC(P)$
	with $\pi < \eta < \mu$. By definition of the refinement order, this means that $\mu$ is
	obtained by removing some blocks $R_1,\ldots,R_k$ from $\pi$ and replacing them with the union of 
	those blocks, where $k\geq 2$.
	If we can show that $k=2$, then we know that $\rho(\pi) + 1 = \rho(\mu)$ and therefore $\rho$
	is a rank function for $\NC(P)$.
	
	Consider the removed block $R_1$. One of the sides of the convex hull $\conv(R_1)$ must
	belong to a line $L$ such that some other $\conv(R_i)$ lies on the opposite side, i.e. 
	so that $L$ separates the interiors of $\conv(R_1)$ and $\conv(R_i)$. Moreover, because all of
	$P$ lies on the boundary of a convex polygon, the convex hull of each other block lies 
	entirely within one of the half-planes bounded by $L$. Without loss of generality, 
	suppose that the interiors of the convex hulls of $R_1,\ldots,R_{i-1}$ lie on one side
	of $L$ and that the interiors of the convex hulls of $R_i,\ldots,R_k$ lie on the other.
	
	Now, define $\eta \in \NC(P)$ from $\pi$ by replacing
	$R_1,\ldots,R_{i-1}$ with their union and replacing $R_i,\ldots,R_k$ with their union.
	It is then clear that $\pi \leq \eta < \mu$, with equality if and only if $k=2$. 
	Since we assumed that $\pi < \mu$ was a covering relation, we therefore must have $k=2$
	and the proof is complete.
\end{proof}

Other potential properties of $\NC(P)$ require a bit of introduction.

\begin{defn}
First, $\NC(P)$ is \emph{self-dual} if there is a bijection $f$ from  $\NC(P)$ to itself such that
$\pi \leq \mu$ if and only if $f(\mu) \leq f(\pi)$. If $\NC(P)$ is graded and 
the number of elements with rank $k$ is equal to the number of elements with rank $n-k-1$
for all $k \in \{0,\ldots,n-1\}$, then $\NC(P)$ is \emph{rank-symmetric}.
Similarly, if $A$ is a subposet of $\NC(P)$ such that the number of elements of rank $k$ in $A$
is equal to the number of elements of rank $n-k-1$ in $A$ (both with respect to the rank function 
of $\NC(P)$), then we say that $A$ is \emph{centered}.
Given $\pi \leq \mu$ in $\NC(P)$, the \emph{interval} $[\pi,\mu]$ is defined to be the
subposet $\{\eta \in \NC(P) \mid \pi \leq \eta \leq \mu \}$. 
Also, a totally ordered subposet of $\NC(P)$ is called a \emph{chain};
a chain is \emph{maximal} if it is not properly contained in another chain and \emph{saturated}
if whenever $\pi < \eta < \mu$ in $\NC(P)$ with $\pi$ and $\mu$ in the chain, then
$\eta$ must belong to the chain as well. 
Finally, a graded poset admits a \emph{symmetric chain decomposition} if its elements can be expressed as the
disjoint union of centered saturated chains. 
\end{defn}

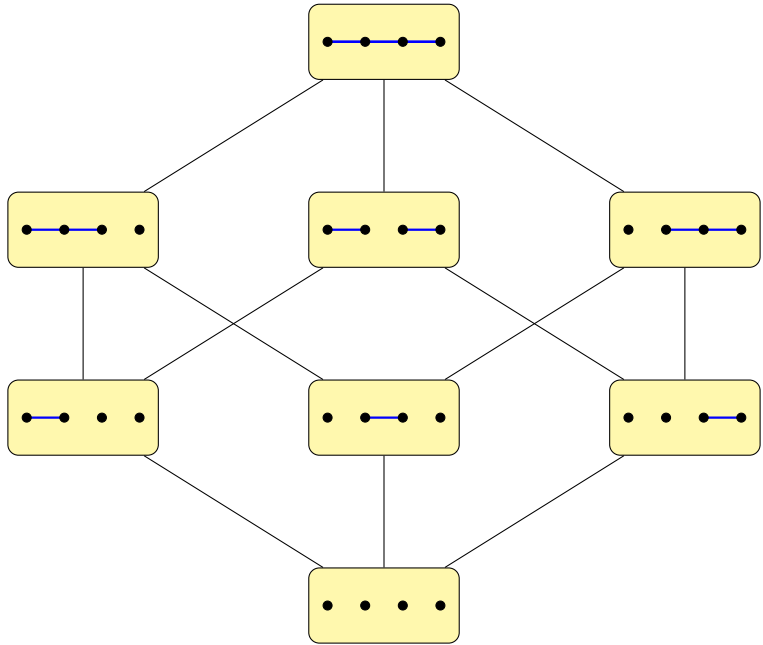
\begin{figure}
	\centering
	\begin{tikzpicture}
        \node[YellowRect] (pa) at (0,0) {};
        
        \foreach \i in {1,2,3} {
            \node[YellowRect] (pb\i) at (-8+4*\i,-2.5) {};
        }
        
        \foreach \i in {1,2,3} {
            \node[YellowRect] (pc\i) at (-8+4*\i,-5) {};
        }
        
        \node[YellowRect] (pd) at (0,-7.5) {};
        
            
        \foreach \i in {1,2,3} {
            \draw (pa) -- (pb\i);
        }
            
        \foreach \i in {1,2,3} {
            \draw (pd) -- (pc\i);
        }    
            
        \draw (pb1) edge (pc1) edge (pc2);
        \draw (pb2) edge (pc1) edge (pc3);
        \draw (pb3) edge (pc2) edge (pc3);
        
        \begin{scope}[shift={(pa)}]
            \makelinepoints
            \filldraw[BluePoly] (1) -- (2) -- (3) -- (4) -- cycle;
            \drawlinepoints
        \end{scope}
        
        \begin{scope}[shift={(pb1)}]
            \makelinepoints
            \draw[BluePoly] (1) -- (2) -- (3);
            \drawlinepoints
        \end{scope}
        
        \begin{scope}[shift={(pb2)}]
            \makelinepoints
            \draw[BluePoly] (1) -- (2);
            \draw[BluePoly] (3) -- (4);
            \drawlinepoints
        \end{scope}
        
        \begin{scope}[shift={(pb3)}]
            \makelinepoints
            \draw[BluePoly] (2) -- (3) -- (4);
            \drawlinepoints
        \end{scope}
        
        \begin{scope}[shift={(pc1)}]
            \makelinepoints
            \draw[BluePoly] (1) -- (2);
            \drawlinepoints
        \end{scope}
        
        \begin{scope}[shift={(pc2)}]
            \makelinepoints
            \draw[BluePoly] (2) -- (3);
            \drawlinepoints
        \end{scope}
        
        \begin{scope}[shift={(pc3)}]
            \makelinepoints
            \draw[BluePoly] (3) -- (4);
            \drawlinepoints
        \end{scope}
        
        
        \begin{scope}[shift={(pd)}]
            \makelinepoints
            \drawlinepoints
        \end{scope}
    \end{tikzpicture}	
    \caption{The lattice of noncrossing partitions $\NC(P_4)$ defined in
    Example~\ref{ex:boolean} is isomorphic to the Boolean lattice $\bool(3)$.}
    \label{fig:ncp-boolean}
\end{figure}

The two most natural examples of $\NC(P)$ arise when the configuration $P$ either lies
on a line or on a circle, and each of these produces a familiar lattice with useful properties.
\begin{example}\label{ex:boolean}
	If $P_n$ consists of points $p_1,\ldots,p_n$ which are arranged in order along a line,
	then each noncrossing partition of $P_n$ is determined by a choice for each 
	$i \in \{1,\ldots,n-1\}$ of whether $p_i$ and $p_{i+1}$ belong to the same block. 
	This provides an isomorphism between $\NC(P_n)$ and the poset of all subsets of 
	$\{1,\ldots,n-1\}$ under inclusion, known as the \emph{Boolean lattice} $\bool(n-1)$. 
	See Figure~\ref{fig:ncp-boolean} for an illustration. Note in particular that $\NC(P_n)$
	is graded and self-dual (and therefore rank-symmetric). Moreover, the size of 
	$\NC(P_n)$ is $2^{n-1}$, and the generating function for this sequence is 
	\[B(x) = \sum_{n\geq 0} 2^{n-1} x^n = \frac{x}{1-2x}.\]
	Finally, the Boolean lattice (and more generally any product of chains) admits a symmetric
	chain decomposition \cite{de-bruijn}.
\end{example}

\begin{example}
	If the configuration $Q_n$ is the vertex set of a convex $n$-gon, then
	$\NC(Q_n)$ is the classical \emph{lattice of noncrossing partitions} $\NC(n)$.
	This lattice is graded, rank-symmetric, and counted by the $n$-th \emph{Catalan number} 
	$C_n = \frac{1}{n+1} \binom{2n}{n}$ \cite{kreweras72}; it is also
	self-dual and admits a symmetric chain decomposition \cite{simion-ullman91}.
	Moreover, the generating function for this sequence is 
	\[
		C(x) = \sum_{n\geq 0} C_n x^n = \frac{1 - \sqrt{1-4x}}{2x}.
	\]
	See Figure~\ref{fig:ncp-classical} for an illustration of $\NC(4)$.
\end{example}

\begin{figure}
	\centering
	\begin{tikzpicture}
        \node [shape=circle,draw,fill=yellow!40,inner sep = 0.6cm] (pa) at (0,0) {};
        
        \foreach \i in {1,...,6} {
            \node [shape=circle,draw,fill=yellow!40,inner sep = 0.6cm] (pb\i) at (-7+2*\i,-2.5) {};
        }
        
        \foreach \i in {1,...,6} {
            \node [shape=circle,draw,fill=yellow!40,inner sep = 0.6cm] (pc\i) at (-7+2*\i,-5.5) {};
        }
        
        \node [shape=circle,draw,fill=yellow!40,inner sep = 0.6cm] (pd) at (0,-8) {};
        
            
        \foreach \i in {1,...,6} {
            \draw (pa) -- (pb\i.north);
        }
            
        \foreach \i in {1,...,6} {
            \draw (pd) -- (pc\i.south);
        }    
            
        \draw (pb1.south) edge (pc1.north) edge (pc2.north);
        \draw (pb2.south) edge (pc1.north) edge (pc3.north) edge (pc5.north);
        \draw (pb3.south) edge (pc2.north) edge (pc3.north) edge (pc6.north);
        \draw (pb4.south) edge (pc2.north) edge (pc4.north) edge (pc5.north);
        \draw (pb5.south) edge (pc1.north) edge (pc4.north) edge (pc6.north);
        \draw (pb6.south) edge (pc5.north) edge (pc6.north);
        
        \begin{scope}[shift={(0,0)}]
            \makepoints
            \filldraw[BluePoly] (1) -- (2) -- (3) -- (4) -- cycle;
            \drawpoints
        \end{scope}
        
        \begin{scope}[shift={(-5,-2.5)}]
            \makepoints
            \draw[BluePoly] (2) -- (3);
            \draw[BluePoly] (1) -- (4);
            \drawpoints
        \end{scope}
        
        \begin{scope}[shift={(-3,-2.5)}]
            \makepoints
            \filldraw[BluePoly] (2) -- (3) -- (4) -- cycle;
            \drawpoints
        \end{scope}
        
        \begin{scope}[shift={(-1,-2.5)}]
            \makepoints
            \filldraw[BluePoly] (1) -- (2) -- (4) -- cycle;
            \drawpoints
        \end{scope}
        
        \begin{scope}[shift={(1,-2.5)}]
            \makepoints
            \filldraw[BluePoly] (1) -- (3) -- (4) -- cycle;
            \drawpoints
        \end{scope}
        
        \begin{scope}[shift={(3,-2.5)}]
            \makepoints
            \filldraw[BluePoly] (1) -- (2) -- (3) -- cycle;
            \drawpoints
        \end{scope}
        
        \begin{scope}[shift={(5,-2.5)}]
            \makepoints
            \draw[BluePoly] (1) -- (2);
            \draw[BluePoly] (3) -- (4);
            \drawpoints
        \end{scope}
        
        \begin{scope}[shift={(-5,-5.5)}]
            \makepoints
            \draw[BluePoly] (2) -- (3);
            \drawpoints
        \end{scope}
        
        \begin{scope}[shift={(-3,-5.5)}]
            \makepoints
            \draw[BluePoly] (1) -- (4);
            \drawpoints
        \end{scope}
        
        \begin{scope}[shift={(-1,-5.5)}]
            \makepoints
            \draw[BluePoly] (2) -- (4);
            \drawpoints
        \end{scope}
        
        \begin{scope}[shift={(1,-5.5)}]
            \makepoints
            \draw[BluePoly] (1) -- (3);
            \drawpoints
        \end{scope}
        
        \begin{scope}[shift={(3,-5.5)}]
            \makepoints
            \draw[BluePoly] (3) -- (4);
            \drawpoints
        \end{scope}
        
        \begin{scope}[shift={(5,-5.5)}]
            \makepoints
            \draw[BluePoly] (1) -- (2);
            \drawpoints
        \end{scope}
        
        
        \begin{scope}[shift={(0,-8)}]
            \makepoints
            \drawpoints
        \end{scope}
    \end{tikzpicture}
    \caption{The classical lattice of noncrossing partitions $\NC(4)$}
    \label{fig:ncp-classical}	
\end{figure}
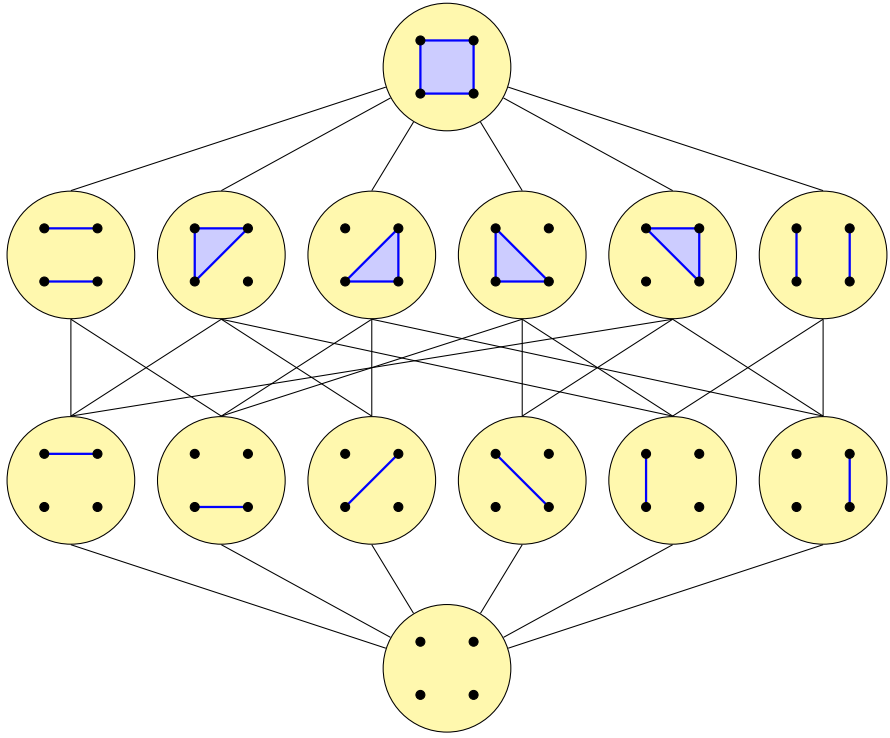

We close this section with an example which belongs to all three of the
configuration types introduced in Definition~\ref{def:three-configurations}.
Let $T_n$ denote a configuration of $n+1$ points $x_1,\ldots,x_n,y$ in $\C$
in which $x_1,\ldots,x_n$ lie in order on a common line and the remaining point $y$ does not. 
To address the recursive structure of $\NC(T_n)$, we will need a technical lemma.

\begin{lem}\label{lem:remove-point}
	Let $P = \{p_1,\ldots,p_n\}$ be a configuration in $\mathbb{C}$ which lies on the boundary
	of a convex polygon in counterclockwise order, and suppose that $p_n$ and $p_1$ are vertices
	of the polygon. 
	\begin{enumerate}
		\item The subposet of $\NC(P)$ consisting of all partitions where
		$p_{n-1}$ and $p_n$ share a block is isomorphic to $\NC(P - \{p_n\})$.
		\item The subposet of $\NC(P)$ consisting of all partitions where
		either $p_{n-1}$ and $p_n$ share a block or $p_n$ is a singleton 
		is isomorphic to the direct product  $\NC(P - \{p_n\}) \times \NC(\{p_{n-1},p_n\})$.
	\end{enumerate}
\end{lem}

\begin{proof}
	For the first claim, define $\alpha = \{\{p_1\},\ldots,\{p_{n-2}\},\{p_{n-1},p_n\}\}$
	and observe that the described subposet is the interval $[\alpha,\hat{1}]$ in $\NC(P)$.
	For each partition $\pi$ in $\NC(P-\{p_n\})$, define $f(\pi)$ simply by adding $p_n$ to
	the block containing $p_{n-1}$. Since $p_n$ and $p_1$ are vertices of the polygon containing $P$ 
	in its boundary, we can see that $f(\pi)$ is also noncrossing, and in particular an element
	of $[\alpha,\hat{1}]$. The fact that $f$ is an order-preserving bijection with order-preserving
	inverse then follows immediately from the definition, so the first claim is proved.
	
	Next, define $\beta = \{\{p_1,\ldots,p_{n-1}\},\{p_n\}\}$ and observe that 
	$\beta \in \NC(P)$ since we assumed that $p_n$ is a vertex of the polygon. Then the
	subposet described in the second claim is the disjoint union of the intervals
	$[\alpha,\hat{1}]$ and $[\hat{0},\beta]$. Since $p_n$ is not in the convex hull of
	$\{p_1,\ldots,p_{n-1}\}$, we know that $[\hat{0},\beta]$ is isomorphic to $\NC(P-\{p_n\})$,
	which we just showed is also isomorphic to $[\alpha,\hat{1}]$.
	Finally, observe that each partition in $[\alpha,\hat{1}]$ covers a unique element
	of $[\hat{0},\beta]$ which is obtained by removing $p_n$ from its block and making it 
	into a singleton. Therefore, the subposet $[\alpha,\hat{1}] \cup [\hat{0},\beta]$ is the direct product of $\NC(P-\{p_n\})$
	and $\NC(\{p_{n-1},p_n\})$.
\end{proof}

\begin{figure}
	\centering
	\begin{tikzpicture}[scale=0.96]
        \node [shape=circle,draw,fill=yellow!40,inner sep = 0.35cm] (pa) at (-7.15+1.1*8,0) {};
        
        \foreach \i in {3,6,7,8,9,10} {
            \node [shape=circle,draw,fill=yellow!40,inner sep = 0.35cm] (pb\i) at (-7.15+1.1*\i,-2) {};
        }
        \node [shape=circle,draw,fill=yellow!40,inner sep = 0.35cm] (pbb) at (-7.15+1.1*11.5,-2) {};

        \foreach \i in {1,...,12} {
            \node [shape=circle,draw,fill=yellow!40,inner sep = 0.35cm] (pc\i) at (-7.15+1.1*\i,-4) {};
        }
        
        \foreach \i in {1,2,3,4,5,8} {
            \node [shape=circle,draw,fill=yellow!40,inner sep = 0.35cm] (pd\i) at (-7.15+1.1*\i,-6) {};
        }
        \node [shape=circle,draw,fill=yellow!40,inner sep = 0.35cm] (pdd) at (-7.15+1.1*11.5,-6) {};
        
       	\node [shape=circle,draw,fill=yellow!40,inner sep = 0.35cm] (pe) at (-7.15+1.1*3,-8) {};
        
        
        \draw[GrayLine, thin] (pa.south) -- (pbb.north);
        \draw[GrayLine, thin] (pbb.south) -- (pc3.north);
        \draw[GrayLine, thin] (pb9.south) -- (pc11.north);
        \draw[GrayLine, thin] (pb10.south) -- (pc12.north);
        \draw[GrayLine, thin] (pc11.south) -- (pd2.north);
        \draw[GrayLine, thin] (pc12.south) -- (pd3.north);
        \draw[GrayLine, thin] (pc10.south) -- (pdd.north);
        \draw[GrayLine, thin] (pdd.south) -- (pe.north);

        \foreach \i in {6,7,8,9,10} {
        	\draw[RedLine, thick] (pa.south) -- (pb\i.north);
        	\draw[RedLine, thick] (pc\i.south) -- (pd8.north);
        	
        }
        
        \begin{scope}[thick, line cap = round, join = round]
        \draw[RedLine, thick] (pb6.south) edge (pc6.north) edge (pc7.north) edge (pc8.north);
        \draw[RedLine, thick] (pb7.south) edge (pc6.north) edge (pc9.north);
        \draw[RedLine, thick] (pb8.south) edge (pc7.north) edge (pc9.north);
        \draw[RedLine, thick] (pb9.south) edge (pc7.north) edge (pc10.north);
        \draw[RedLine, thick] (pb10.south) edge (pc8.north) edge (pc9.north) edge (pc10.north);
        
        \draw[PurpleLine, thick, line cap = round] (pa.south) -- (pb3.north);
        \draw[PurpleLine, thick] (pb6.south) -- (pc1.north);
        \draw[PurpleLine, thick] (pb7.south) -- (pc2.north);
        \draw[PurpleLine, thick] (pb8.south) -- (pc3.north);
        \draw[PurpleLine, thick] (pb9.south) -- (pc4.north);
        \draw[PurpleLine, thick] (pb10.south) -- (pc5.north);
        \draw[PurpleLine, thick] (pc6.south) -- (pd1.north);
        \draw[PurpleLine, thick] (pc7.south) -- (pd2.north);
        \draw[PurpleLine, thick] (pc8.south) -- (pd3.north);
        \draw[PurpleLine, thick] (pc9.south) -- (pd4.north);
        \draw[PurpleLine, thick] (pc10.south) -- (pd5.north);
        \draw[PurpleLine, thick] (pd8.south) -- (pe.north);
        
        \draw[GreenLine, thick] (pc11.north) -- (pbb.south) -- (pc12.north);
        \draw[GreenLine, thick] (pc11.south) -- (pdd.north) -- (pc12.south);
        
		\foreach \i in {1,2,3,4,5} {
        	\draw[RedLine, thick] (pb3.south) -- (pc\i.north);
        	\draw[RedLine, thick] (pd\i.south) -- (pe.north);
        }
        
        \draw[RedLine, thick] (pc1.south) edge (pd1.north) edge (pd2.north) edge (pd3.north);
        \draw[RedLine, thick] (pc2.south) edge (pd1.north) edge (pd4.north);
        \draw[RedLine, thick] (pc3.south) edge (pd2.north) edge (pd4.north);
        \draw[RedLine, thick] (pc4.south) edge (pd2.north) edge (pd5.north);
        \draw[RedLine, thick] (pc5.south) edge (pd3.north) edge (pd4.north) edge (pd5.north);
        \end{scope}
        
        \begin{scope}[shift={(pa)}]
            \maketpoints
            \filldraw[BluePoly, join = bevel] (0) -- (1) -- (2) -- (3) -- (4) -- cycle;
            \drawtpoints
        \end{scope}
        
        \begin{scope}[shift={(pb3)}]
        	\maketpoints
            \filldraw[BluePoly] (0) -- (1) -- (2) -- (3) -- cycle;
            \drawtpoints
        \end{scope}
        
        \begin{scope}[shift={(pb6)}]
        	\maketpoints
            \filldraw[BluePoly, join = bevel] (0) -- (1) -- (2) -- cycle;
            \draw[BlueLine] (3) -- (4);
            \drawtpoints
        \end{scope}
        
        \begin{scope}[shift={(pb7)}]
        	\maketpoints
            \draw[BlueLine] (0) -- (1);
            \draw[BlueLine] (2) -- (3) -- (4);
            \drawtpoints
        \end{scope}
        
        \begin{scope}[shift={(pb8)}]
        	\maketpoints
            \draw[BlueLine] (1) -- (2) -- (3) -- (4);
            \drawtpoints
        \end{scope}
        
        \begin{scope}[shift={(pb9)}]
        	\maketpoints
            \filldraw[BluePoly, join = bevel] (0) -- (3) -- (4) -- cycle;
            \draw[BlueLine] (1) -- (2);
            \drawtpoints
        \end{scope}
        
        \begin{scope}[shift={(pb10)}]
        	\maketpoints
            \filldraw[BluePoly, join = bevel] (0) -- (2) -- (3) -- (4) -- cycle;
            \drawtpoints
        \end{scope}
        
        \begin{scope}[shift={(pbb)}]
        	\maketpoints
            \draw[BlueLine] (0) -- (4);
            \draw[BlueLine] (1) -- (2) -- (3);
            \drawtpoints
        \end{scope}

		
		\begin{scope}[shift={(pc1)}]
        	\maketpoints
            \filldraw[BluePoly, join=bevel] (0) -- (1) -- (2) -- cycle;
            \drawtpoints
        \end{scope}
        
        \begin{scope}[shift={(pc2)}]
        	\maketpoints
            \draw[BlueLine] (0) -- (1);
            \draw[BlueLine] (2) -- (3);
            \drawtpoints
        \end{scope}
        
        \begin{scope}[shift={(pc3)}]
        	\maketpoints
            \draw[BlueLine] (1) -- (2) -- (3);
            \drawtpoints
        \end{scope}
        
        \begin{scope}[shift={(pc4)}]
        	\maketpoints
            \draw[BlueLine] (0) -- (3);
            \draw[BlueLine] (1) -- (2);
            \drawtpoints
        \end{scope}
        
        \begin{scope}[shift={(pc5)}]
        	\maketpoints
            \filldraw[BluePoly, join=bevel] (0) -- (2) -- (3) -- cycle;
            \drawtpoints
        \end{scope}
        
        \begin{scope}[shift={(pc6)}]
        	\maketpoints
            \draw[BlueLine] (0) -- (1);
            \draw[BlueLine] (3) -- (4);
            \drawtpoints
        \end{scope}
        
        \begin{scope}[shift={(pc7)}]
        	\maketpoints
            \draw[BlueLine] (1) -- (2);
            \draw[BlueLine] (3) -- (4);
            \drawtpoints
        \end{scope}
        
        \begin{scope}[shift={(pc8)}]
        	\maketpoints
            \draw[BlueLine] (0) -- (2);
            \draw[BlueLine] (3) -- (4);
            \drawtpoints
        \end{scope}
        
        \begin{scope}[shift={(pc9)}]
        	\maketpoints
            \draw[BlueLine] (2) -- (3) -- (4);
            \drawtpoints
        \end{scope}
        
        \begin{scope}[shift={(pc10)}]
        	\maketpoints
            \filldraw[BluePoly, join=bevel] (0) -- (3) -- (4) -- cycle;
            \drawtpoints
        \end{scope}
        
        \begin{scope}[shift={(pc11)}]
        	\maketpoints
            \draw[BlueLine] (0) -- (4);
            \draw[BlueLine] (1) -- (2);
            \drawtpoints
        \end{scope}
        
        \begin{scope}[shift={(pc12)}]
        	\maketpoints
            \draw[BlueLine] (0) -- (4);
            \draw[BlueLine] (2) -- (3);
            \drawtpoints
        \end{scope}
        
        \begin{scope}[shift={(pd1)}]
        	\maketpoints
            \draw[BlueLine] (0) -- (1);
            \drawtpoints
        \end{scope}
        
        \begin{scope}[shift={(pd2)}]
        	\maketpoints
            \draw[BlueLine] (1) -- (2);
            \drawtpoints
        \end{scope}
        
        \begin{scope}[shift={(pd3)}]
        	\maketpoints
            \draw[BlueLine] (0) -- (2);
            \drawtpoints
        \end{scope}
        
        \begin{scope}[shift={(pd4)}]
        	\maketpoints
            \draw[BlueLine] (2) -- (3);
            \drawtpoints
        \end{scope}
        
        \begin{scope}[shift={(pd5)}]
        	\maketpoints
            \draw[BlueLine] (0) -- (3);
            \drawtpoints
        \end{scope}
        
        \begin{scope}[shift={(pd8)}]
        	\maketpoints
            \draw[BlueLine] (3) -- (4);
            \drawtpoints
        \end{scope}
        
        \begin{scope}[shift={(pdd)}]
        	\maketpoints
            \draw[BlueLine] (0) -- (4);
            \drawtpoints
        \end{scope}
        
        \begin{scope}[shift={(pe)}]
        	\maketpoints
            \drawtpoints
        \end{scope}
       
    \end{tikzpicture}
    \caption{The noncrossing partition lattice $\NC(T_4)$. The edges have been colored
    so that the red edges and purple edges illustrate the subposet $A$
    defined in Lemma~\ref{lem:t-partition-recursive} and the green edges illustrate
    the subposet $B$.}
    \label{fig:ncp-t}
\end{figure}
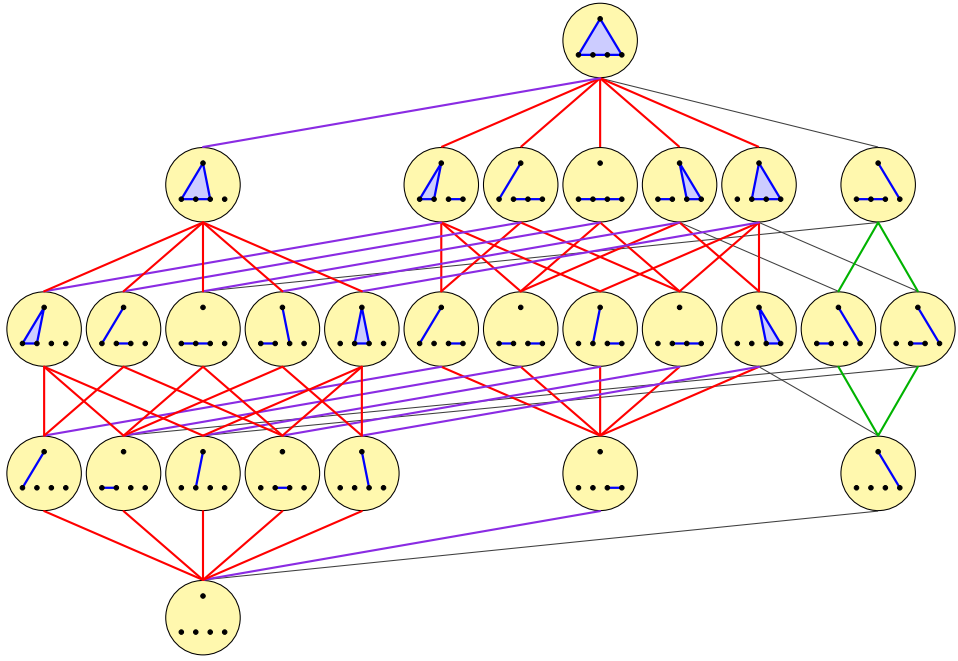

We are now ready to identify some recursive structure for $\NC(T_n)$. See
Figure~\ref{fig:ncp-t} for an illustration.

\begin{lem}\label{lem:t-partition-recursive}
	Let $n\geq 2$, let $A$ be the subposet of $\NC(T_n)$ which consists of all partitions in which
	$x_n$ is either a singleton or in a block with $x_{n-1}$, and let $B$ be the subposet
	of $\NC(T_n)$ consisting of all partitions in which $\{x_n,y\}$ is a block. Then
	\begin{enumerate}
		\item $A$ is isomorphic to $\NC(T_{n-1}) \times \bool(1)$;
		\item $B$ is isomorphic to $\bool(n-2)$.
	\end{enumerate}
	Moreover, $\NC(T_n)$ is the disjoint union of $A$ and $B$.
\end{lem}

\begin{proof}
	The first claim follows from Lemma~\ref{lem:remove-point}. The second claim is 
	straightforward since the convex hull of $\{x_n,y\}$ is disjoint from the
	convex hull of $\{x_1,\ldots,x_{n-1}\}$, and $\NC(\{x_1,\ldots,x_{n-1}\})$ is isomorphic
	to $\bool(n-2)$. For the final claim, note that the 
	arrangement of the configuration $T_n$ tells us that $x_n$
	must either be in a block by itself, be in a block with $y$ and nothing else, or 
	share a block with $x_{n-1}$ (and possibly other points). So the claim that 
	$\NC(T_n)$ is the disjoint union of $A$ and $B$ is immediate and the proof is complete.
\end{proof}

Next, we can address the presence of a symmetric chain decomposition. The following observation
will be useful at several points in the rest of the article.

\begin{rem}\label{rem:scd}
	Suppose $\NC(P)$ is graded with rank function $\rho$, and suppose that $\NC(P)$ is the disjoint
	union of subposets $A_1,\ldots,A_k$ such that each $A_i$ is a bounded poset with minimum element
	$\hat{0}_i$ and maximum element $\hat{1}_i$. If 
	\begin{enumerate}
		\item each $A_i$ is a union of intervals in $\NC(P)$,
		\item each $A_i$ admits a symmetric chain decomposition, and
		\item $\rho(\hat{0}) + \rho(\hat{1}) = \rho(\hat{0}_i) + \rho(\hat{1}_i)$ for all 
		$i\in\{1,\ldots,k\}$,
	\end{enumerate}
	then $\NC(P)$ is a union of centered subposets with symmetric chain decompositions which
	can be combined to form a symmetric chain decomposition for $\NC(P)$.
\end{rem}

\begin{prop}\label{prop:t-scd}
	$\NC(T_n)$ admits a symmetric chain decomposition.
\end{prop}

\begin{proof}
	First, we know from Theorem~\ref{thm:graded} that $\NC(T_n)$ is graded.
	As base cases, note that $\NC(T_0)$ and $\NC(T_1)$ consist of only one element and two elements
	respectively, and they trivially have symmetric chain decompositions as they each 
	consist of a single chain. Proceeding by induction, let $n\geq 2$ and suppose that 
	$\NC(T_{n-1})$ has a symmetric chain decomposition; in particular, this means that 
	$\NC(T_{n-1})$ is rank-symmetric. We know by 
	Lemma~\ref{lem:t-partition-recursive} that $\NC(T_n)$ is the disjoint union of
	subposets isomorphic to $\NC(T_{n-1})\times \bool(1)$ and $\bool(n-2)$. Products of posets
	with symmetric chain decompositions have symmetric chain decompositions \cite{de-bruijn},
	so the subposets $A$ and $B$ meet the criteria for Remark~\ref{rem:scd}. Therefore, $\NC(T_n)$ 
	has a symmetric chain decomposition as well, and we are done.
\end{proof}

\begin{prop}\label{prop:t-gen-func}
	If $t_n = |\NC(T_n)|$, then the associated generating function is
	\[
		T(x) = \sum_{n\geq 0} t_n x^n = \frac{(1-x)^2}{(1-2x)^2}.
	\]
\end{prop}

\begin{proof}
	When $n\geq 2$, we know that $t_n$ satisfies the recursive formula
	\[
		t_n = 2t_{n-1} + 2^{n-2}
	\]	
	by Lemma~\ref{lem:t-partition-recursive},
	and the sequence begins with the terms $1, 2, 5, 12, 28, 64,\ldots$ for $n\geq 0$. 
	Therefore, we have 
	\begin{align*}
		T(x) &= \sum_{n\geq 0} t_n x^n \\
		&= 1 + 2x + \sum_{n\geq 2} t_n x^n \\
		&= 1 + 2x + 2 \sum_{n\geq 2}	t_{n-1} x^n + \sum_{n\geq 2} 2^{n-2} x^n \\
		&= 1 + 2x + 2x \sum_{n\geq 1} t_n x^n + x^2 \sum_{n\geq 0} 2^n x^n \\
		&= 1 + 2x + 2x (T(x) - 1) + x^2 \frac{1}{1-2x} \\
		T(x)(1-2x) &= 1+x^2\frac{1}{1-2x} \\
		T(x)(1-2x) &= \frac{1-2x+x^2}{1-2x} \\
		T(x) &= \frac{(1-x)^2}{(1-2x)^2}
	\end{align*}
	and we are done.
\end{proof}

As a final note, we observe that there is a closed form for the recursive 
formula given in the proof of Proposition~\ref{prop:t-gen-func}: 
$t_n = (n+3)2^{n-2}$ when $n \geq 2$. To see this, observe as a base case 
that $t_2 = 5$, then assume that for some $n \geq 2$, the formula works for $t_n$; 
then we have
\begin{align*}
	t_{n+1} &= 2t_n + 2^{n-1} \\
	&= 2(n+3)2^{n-2} + 2^{n-1} \\
	&= (n+3)2^{n-1} + 2^{n-1} \\
	&= (n+4)2^{n-1}
\end{align*}
as desired. 

\section{Cone configurations}\label{sec:cone-config}

Recall the definitions of open cone configurations $U_{m,n}$ and closed cone 
configurations $V_{m,n}$ from Definition~\ref{def:three-configurations}.
In this section we will show that the corresponding lattices $\NC(U_{m,n})$ and $\NC(V_{m,n})$
admit symmetric chain decompositions and provide multivariate generating functions which 
count the number of elements in each.
For the remainder of this section, assume that $U_{m,n}$ consists of $m$ points 
$x_1,\ldots,x_m$ in order on the positive real axis and $n$ points $y_1,\ldots,y_n$ 
in order on the positive imaginary axis in $\mathbb{C}$, and let $V_{m,n}$ be the same 
configuration with the additional point $z$ at the origin.

\begin{lem}\label{lem:u-partition-recursive}
	Let $m \geq 2$ and $n \geq 1$. Define $A$ to be the subposet of $\NC(U_{m,n})$
	which consists of all partitions in which $x_m$ is either a singleton or in a
	block with $x_{m-1}$. For each $k \in \{1,\ldots,n\}$, define $B_k$ to be the 
	subposet of $\NC(U_{m,n})$ consisting of all partitions in which $x_m$ shares a block
	with $y_k$ but not $x_{m-1},y_1,\ldots,y_{k-1}$. Then
	\begin{enumerate}
		\item $A$ is isomorphic to $\NC(U_{m-1,n}) \times \bool(1)$;
		\item $B_k$ is isomorphic to $\NC(U_{m-1,k-1}) \times \bool(n-k)$
	\end{enumerate}
	Moreover, $\NC(U_{m,n})$ is the disjoint union of $A,B_1,\ldots,B_n$.
\end{lem}

\begin{proof}
	The first claim follows from Lemma~\ref{lem:remove-point}. For the second
	claim, we know by Lemma~\ref{lem:remove-point} that $B_k$ is isomorphic to
	$\NC(U_{m-1,k-1}) \times \NC(U_{0,n-k+1})$, and we know that the second factor
	is isomorphic to $\bool(n-k)$. Finally, we know that in each partition of $U_{m,n}$, 
	the point $x_m$ must either be a singleton or share a block with $x_{m-1}$ or one of
	$y_1,\ldots,y_n$, so $\NC(U_{m,n})$ is the union of $A,B_1,\ldots,B_n$, and it
	is clear from the definition that these are disjoint.
\end{proof}

Note that one may replace $U_{m,n}$ in the proof above with $V_{m,n}$ to obtain
an identical result, as the additional point $z$ does not require any extra posets
for the decomposition. Thus, we can use Lemma~\ref{lem:u-partition-recursive}
for both $\NC(U_{m,n})$ and $\NC(V_{m,n})$. As a first application, we use this recursive
structure to show that both lattices have symmetric chain decompositions.

\begin{thm}\label{thm:uv-scd}
	$\NC(U_{m,n})$ and $\NC(V_{m,n})$ admit symmetric chain decompositions.	
\end{thm}

\begin{proof}
	First, observe that $\NC(U_{0,n})$ and $\NC(U_{m,0})$ 
	are Boolean lattices and therefore admit symmetric chain decompositions. We also have
	$\NC(U_{1,n}) \cong \NC(T_n)$, which has a symmetric chain 
	decomposition by Proposition~\ref{prop:t-scd}. 
	
	Proceeding by induction on both $m$ and $n$, suppose that $\NC(U_{a,b})$ has a symmetric
	chain decomposition when $a < m$ and $b\leq n$. 	When $m\geq 2$ and $n \geq 1$, 
	Lemma~\ref{lem:u-partition-recursive} tells us that $\NC(U_{m,n})$ is the disjoint union
	of subposets which are isomorphic to $\NC(U_{m-1,n}) \times \bool(1)$
	and $\NC(U_{m-1,k-1}) \times \bool(n-k)$ for $k \in \{1,\ldots,n\}$. Combining the
	inductive hypothesis with the fact that Boolean lattices have symmetric chain decompositions, 
	we know by Remark~\ref{rem:scd} that $\NC(U_{m,n})$ is the union of disjoint centered
	subposets with symmetric chain decompositions, and therefore $\NC(U_{m,n})$ itself
	admits a symmetric chain decomposition. The same result for $\NC(V_{m,n})$ 
	follows analogously.
\end{proof}

To close this section, we examine the sizes of $\NC(U_{m,n})$ and $\NC(V_{m,n})$. 
See Table~\ref{tab:uvmn} for values with small $m$ and $n$. Note that the size of
$\NC(U_{m,n})$ matches the OEIS sequence \href{https://oeis.org/A035002}{A035002}
and the size of $\NC(V_{m,n})$ matches the sequence 
\href{https://oeis.org/A341867}{A341867} \cite{oeis}.
More concretely, we now provide multivariate generating functions for the sizes 
of $\NC(U_{m,n})$ and $\NC(V_{m,n})$. 

\begin{table}
	\begin{tabular}{c|ccccc}
		\multicolumn{6}{l}{$|\NC(U_{m,n})|$} \\
		\hline \hline 
		$m \backslash n$ & 0 & 1 & 2 & 3 & 4 \\
		\hline
		0 & 0 & 1 & 2 & 4 & 8 \\
		1 & 1 & 2 & 5 & 12 & 28 \\
		2 & 2 & 5 & 14 & 37 & 94 \\
		3 & 4 & 12 & 37 & 106 & 289 \\
		4 & 8 & 28 & 94 & 289 & 838 \\
	\end{tabular}
	\hfill
	\begin{tabular}{c|ccccc}
		\multicolumn{6}{l}{$|\NC(V_{m,n})|$} \\
		\hline \hline 
		$m \backslash n$ & 0 & 1 & 2 & 3 & 4 \\
		\hline
		0 & 1 & 2 & 4 & 8 & 16 \\
		1 & 2 & 5 & 12 & 28 & 64 \\
		2 & 4 & 12 & 33 & 86 & 216 \\
		3 & 8 & 28 & 86 & 245 & 664 \\
		4 & 16 & 64 & 216 & 664 & 1921 \\
	\end{tabular}
	\vspace{1em}
	\caption{Sizes of $\NC(U_{m,n})$ and
	$\NC(V_{m,n})$ for small values of $m$ and $n$}
	\label{tab:uvmn}	
\end{table}

\begin{thm}\label{thm:u-gf}
	If $u_{m,n}$ is the size of $\NC(U_{m,n})$, then the associated multivariate generating
	function is
	\[
		U(x,y) = \sum_{\substack{m,n\geq 0}} u_{m,n}x^m y^n = \frac{x + y - 2xy}{1-2x-2y+3xy}.
	\]
\end{thm}

\begin{proof}
	As a first step, we have 
	\begin{align*}
		U(x,y) &= \sum_{\substack{m\geq 0 \\ n\geq 0}} u_{m,n} x^m y^n \\
		&= \sum_{n\geq 0} u_{0,n} y^n + \sum_{n\geq 0} u_{1,n} xy^n
		+ \sum_{m\geq 2} u_{m,0} x^m + \sum_{\substack{m\geq 2 \\ n\geq 1}} u_{m,n} x^m y^n.
	\end{align*}
	The first three sums are straightforward to deal with. The first is 
	\[
		\sum_{n\geq 0} u_{0,n} y^n = \sum_{n\geq 1} 2^{n-1} y^n = \frac{y}{1-2y}
	\]
	by Example~\ref{ex:boolean}, the second is 
	\[
		\sum_{n\geq 0} u_{1,n} xy^n = x \sum_{n\geq 0} t_n y^n = x\frac{(1-y)^2}{(1-2y)^2}
	\]
	by Proposition~\ref{prop:t-gen-func}, and the third is
	\[
		\sum_{m\geq 2} u_{m,0} x^m = \sum_{m\geq 2} 2^{m-1} x^m = \frac{2x^2}{1-2x}
	\]
	by Example~\ref{ex:boolean} again. Combining these three sums, we obtain
	\[
		\sum_{n\geq 0} u_{0,n} y^n + \sum_{n\geq 0} u_{1,n} xy^n + \sum_{m\geq 2} u_{m,0} x^m 
		= \frac{y}{1-2y} + x\frac{(1-y)^2}{(1-2y)^2} + \frac{2x^2}{1-2x}.
	\]
	As for the fourth sum, we know by Lemma~\ref{lem:u-partition-recursive} that
	\[
		u_{m,n} = 2u_{m-1,n} + \sum_{k=1}^n u_{m-1,k-1} 2^{n-k}
	\]	
	when $m \geq 2$ and $n \geq 1$. Applying this recursive formula, we 
	can write the fourth sum as follows:
	\begin{align*}
		\sum_{\substack{m\geq 2 \\ n\geq 1}} u_{m,n} x^m y^n &= 
		2 \sum_{\substack{m\geq 2 \\ n\geq 1}} u_{m-1,n} x^m y^n + 
		\sum_{\substack{m\geq 2 \\ n\geq 1}} \sum_{k=1}^n u_{m-1,k-1} 2^{n-k} x^m y^n \\
		&= 2x \sum_{\substack{m\geq 1 \\ n \geq 1}} u_{m,n} x^m y^n + 
		xy\sum_{\substack{m\geq 1 \\ n\geq 1}} \sum_{k=0}^{n-1} u_{m,k} 2^{n-k-1} x^m y^{n-1} \\
		&= 2x \sum_{\substack{m\geq 1 \\ n \geq 1}} u_{m,n} x^m y^n + 
		xy\left(\sum_{\substack{m\geq 1 \\ n \geq 0}} u_{m,n} x^m y^n\right)
		\left(\sum_{n\geq 0} 2^{n} y^n\right) \\ 
		&= 2x\left(U(x,y) - \frac{x}{1-2x} - \frac{y}{1-2y}\right) + \left(U(x,y) 
		- \frac{y}{1-2y}\right)\frac{xy}{1-2y}\\
		&= U(x,y)\left(2x + \frac{xy}{1-2y}\right) - \frac{2x^2}{1-2x} - \frac{2xy}{1-2y}
		- \frac{x y^2}{(1-2y)^2}.
	\end{align*}
	Putting it all together, we observe that two terms immediately cancel and we have
	\begin{align*}
		U(x,y) &= U(x,y) \frac{2x-3xy}{1-2y} + \frac{y}{1-2y} + x\frac{(1-y)^2}{(1-2y)^2}
		- \frac{2xy}{1-2y} - \frac{x y^2}{(1-2y)^2},
	\end{align*}
	which can be rearranged to get
	\begin{align*}
		U(x,y)\left(1 - 	\frac{2x-3xy}{1-2y}\right) &= \frac{y-2xy}{1-2y} + \frac{x(1-y)^2 - xy^2}{(1-2y)^2} \\
		U(x,y)\frac{1-2x-2y+3xy}{1-2y} &= \frac{y(1-2x)}{1-2y} + \frac{x(1-2y)}{(1-2y)^2} \\
		U(x,y)(1-2x-2y+3xy) &= x + y -2xy \\
		U(x,y) &= \frac{x+y-2xy}{1-2x-2y+3xy}
	\end{align*}
	and we are done.
\end{proof}

\begin{thm}\label{thm:v-gf}
	If $v_{m,n}$ is the size of $\NC(V_{m,n})$, then the associated multivariate generating
	function is
	\[
		V(x,y) = \sum_{\substack{m,n\geq 0}} v_{m,n}x^m y^n = \frac{1}{1-2x-2y+3xy}.
	\]
\end{thm}

\begin{proof}
	First, we have
	\begin{align}
		V(x,y) &= \sum_{\substack{m\geq 0 \\ n\geq 0}} v_{m,n} x^m y^n \nonumber \\
		&= v_{0,0} + \sum_{n\geq 1} v_{0,n} y^n + \sum_{m \geq 1} v_{m,0} x^m + 
		\sum_{\substack{m\geq 1 \\ n\geq 1}} v_{m,n} x^m y^n \nonumber \\
		&= 1 + \frac{2y}{1-2y} + \frac{2x}{1-2x} + 
		\sum_{\substack{m\geq 1 \\ n\geq 1}} v_{m,n} x^m y^n. \label{eq:v-gf}
	\end{align}
	Recalling that $V_{m,n}$ is obtained from $U_{m,n}$ by including the additional
	corner point $z$, we observe that $z$ is either a singleton or it shares a block with 
	$x_1$, $y_1$, or both. When $m$ and $n$ are at least $1$, we thus have the
	recursive enumeration formula:
	\[
		v_{m,n} = u_{m,n} + v_{m-1,n} + v_{m,n-1} - v_{m-1,n-1}.
	\]
	The first term on the righthand side corresponds to the case where $z$ is a singleton. 
	The second and third terms correspond to the case where $z$ shares a block with $x_1$
	or $y_1$ respectively, and we subtract the fourth term to account for the overlap, in which
	$z$ shares a block with both $x_1$ and $y_1$.

	Returning to the generating function, we can apply the recursive enumeration formula 
	to write the remaining sum as
	\[
		\sum_{\substack{m\geq 1 \\ n\geq 1}} u_{m,n} x^m y^n
		+ \sum_{\substack{m\geq 1 \\ n\geq 1}} v_{m-1,n} x^m y^n 
		+ \sum_{\substack{m\geq 1 \\ n\geq 1}} v_{m,n-1} x^m y^n
		- \sum_{\substack{m\geq 1 \\ n\geq 1}} v_{m-1,n-1} x^m y^n,
	\]
	which can be reindexed to provide
	\[		
		\sum_{\substack{m\geq 1 \\ n\geq 1}} u_{m,n} x^m y^n
		+ x\sum_{\substack{m\geq 0 \\ n\geq 1}} v_{m,n} x^m y^n 
		+ y\sum_{\substack{m\geq 1 \\ n\geq 0}} v_{m,n} x^m y^n
		- xy\sum_{\substack{m\geq 0 \\ n\geq 0}} v_{m,n} x^m y^n.
	\]
	We handle each of the four sums in turn:
	\begin{itemize}
		\item $\displaystyle\sum_{\substack{m\geq 1 \\ n\geq 1}} u_{m,n} x^m y^n = 
		U(x,y) - \frac{x}{1-2x} - \frac{y}{1-2y}$;
		\item $\displaystyle x\sum_{\substack{m\geq 0 \\ n\geq 1}} v_{m,n} x^m y^n = 
		xV(x,y) - \frac{x}{1-2x};$
		\item $\displaystyle y\sum_{\substack{m\geq 1 \\ n\geq 0}} v_{m,n} x^m y^n = 
		yV(x,y) - \frac{y}{1-2y};$
		\item $\displaystyle xy\sum_{\substack{m\geq 0 \\ n\geq 0}} v_{m,n} x^m y^n = 
		xy V(x,y).$
	\end{itemize}
	Returning to Equation~\ref{eq:v-gf}, we now have
	\begin{align*}
		V(x,y) &= 1 + U(x,y) + xV(x,y) + yV(x,y) - xyV(x,y) \\
		V(x,y)(1-x-y+xy) &= 1 + \frac{x+y-2xy}{1-2x-2y+3xy}	\\
		V(x,y)(1-x-y+xy) &= \frac{1-x-y+xy}{1-2x-2y+3x} \\
		V(x,y) &= \frac{1}{1-2x-2y+3xy}
	\end{align*}
	as desired.
\end{proof}

\section{Semicircular configurations}\label{sec:semicirc-config}

In our final section, we will examine the noncrossing partition lattices for
semicircular configurations, introduced in Definition~\ref{def:three-configurations}.
In particular, we show that these lattices admit symmetric chain decompositions
and produce the associated multivariate generating functions. For the rest 
of this section, let $S_{m,n}$ denote a configuration of $m+2$ points $x_0,\ldots,x_{m+1}$
on the line segment from $-1$ to $1$ in $\mathbb{C}$ and $n$ points $y_1,\ldots,y_n$ 
on the upper half of the unit circle.

As in the previous section, we begin with a lemma on the recursive structure for $\NC(S_{m,n})$.
Note the similarities with Lemma~\ref{lem:u-partition-recursive}.

\begin{lem}\label{lem:s-partition-recursive}
	Let $m\geq 1$ and $n\geq 1$. Define $A$ to be the subposet of $\NC(S_{m,n})$ which consists
	of all partitions in which $x_{m+1}$ is either a singleton or in a block with $x_m$.
	For each $k \in \{1,\ldots,n\}$, define $B_k$ to be the subposet of $\NC(S_{m,n})$ consisting
	of all partitions in which $x_{m+1}$ shares a block with $y_k$ but not $x_m,y_1,\ldots,y_{k-1}$.
	Then
	\begin{enumerate}
		\item $A$ is isomorphic to $\NC(S_{m-1,n}) \times \bool(1)$;
		\item $B_k$ is isomorphic to $\NC(S_{m-1,k-1}) \times \NC(n-k+1)$.
	\end{enumerate}
	Moreover, $\NC(S_{m,n})$ is the disjoint union of $A,B_1,\ldots,B_k$.
\end{lem}

\begin{proof}
	The first claim follows directly from Lemma~\ref{lem:remove-point}. For the second,
	we know by Lemma~\ref{lem:remove-point} that $B_k$ is isomorphic to 
	$\NC(S_{m-1,k-1}) \times \NC(\{y_k,\ldots,y_n\})$,
	and since the points  $y_k,\ldots,y_n$ are assumed to lie on the boundary of a circle, we know that
	$\NC(\{y_k,\ldots,y_n\})$ is isomorphic to the classical lattice of noncrossing partitions
	$\NC(n-k+1)$. Finally, we know that in each noncrossing partition of $S_{m,n}$, the point
	$x_{m+1}$ must either be a singleton or share a block with $x_m$ or one of $y_1,\ldots,y_n$,
	so $\NC(S_{m,n})$ is the union of $A,B_1,\ldots,B_n$ and it is clear from the definition
	that these subposets are disjoint.	
\end{proof}

Our first application of Lemma~\ref{lem:s-partition-recursive} is to show that $\NC(S_{m,n})$
has a symmetric chain decomposition (and is therefore rank-symmetric).

\begin{thm}\label{thm:s-scd}
	$\NC(S_{m,n})$ admits a symmetric chain decomposition.	
\end{thm}

\begin{proof}
	As base cases, observe that $\NC(S_{0,n})$ is the classical lattice of noncrossing partitions
	$\NC(n)$ and $\NC(S_{m,0})$ is the Boolean lattice $\bool(m+1)$, each of which is known
	to admit a symmetric chain decomposition. 
	Proceeding by induction on both $m$ and $n$, suppose that $\NC(S_{a,b})$ has a symmetric
	chain decomposition when $a < m$ and $b\leq n$. 	When $m\geq 1$ and $n \geq 1$, 
	Lemma~\ref{lem:s-partition-recursive} tells us that $\NC(S_{m,n})$ is the disjoint union of
	centered subposets which are isomorphic to $\NC(S_{m-1,n}) \times \bool(1)$ and
	$\NC(S_{m-1,k-1}) \times \NC(n-k+1)$. Using the inductive hypothesis and the facts that 
	Boolean lattices and classical lattices of noncrossing partitions each have 
	symmetric chain decompositions, we know by Remark~\ref{rem:scd} that
	$\NC(S_{m,n})$ has a symmetric chain decomposition as well.
\end{proof}

\begin{table}
	\centering
	\begin{tabular}{c|ccccc}
		\multicolumn{6}{l}{$|\NC(S_{m,n})|$} \\
		\hline \hline 
		$m \backslash n$ & 0 & 1 & 2 & 3 & 4 \\
		\hline
		0 & 2 & 5 & 14 & 42 & 132 \\
		1 & 4 & 12 & 37 & 118 & 387 \\
		2 & 8 & 28 & 94 & 317 & 1082 \\
		3 & 16 & 64 & 232 & 824 & 2921 \\
		4 & 32 & 144 & 560 & 2088 & 7674 \\
	\end{tabular}
	\vspace{1em}
	\caption{Sizes of $\NC(S_{m,n})$ for small values of $m$ and $n$}
	\label{tab:smn}
\end{table}

In Table~\ref{tab:smn}, we provide the size of $\NC(S_{m,n})$ for small values of $m$
and $n$. We are not aware of any previous appearances of these numbers.
Finally, we establish a multivariate generating function which enumerates the noncrossing 
partitions of $S_{m,n}$. As in the previous section, it would be interesting to find a closed
enumeration formula.

\begin{thm}\label{thm:s-gf}
	If $s_{m,n}$ is the size of $\NC(S_{m,n})$, then the associated multivariate generating
	function is
	\[
		S(x,y) = \sum_{\substack{m,n\geq 0}} s_{m,n}x^m y^n = \left(\frac{C(y)-1-y}{y^2}\right)
		\left(\frac{1}{1-x(1+C(y))}\right),
	\]
	where $C(y) = (1-\sqrt{1-4y})/2y$ is the generating function for the Catalan numbers.
\end{thm}

\begin{proof}
	To start, we write
	\begin{align}
		S(x,y) &= \sum_{\substack{m\geq 0 \\ n\geq 0}} s_{m,n} x^m y^n \nonumber \\
		&= s_{0,0} + \sum_{m\geq 1} s_{m,0} x^m + \sum_{n\geq 1} s_{0,n} y^n +	
		\sum_{\substack{m\geq 1 \\ n\geq 1}} s_{m,n} x^m y^n  \nonumber \\
		&= 2 + \sum_{m\geq 1} 2^{m+1} x^m + \sum_{n\geq 1} C_{n+2}y^n + 
		\sum_{\substack{m\geq 1 \\ n\geq 1}} s_{m,n} x^m y^n  \nonumber \\
		&= 2 + 4x \sum_{m\geq 0}2^m x^m + \sum_{n\geq 3} C_{n} y^{n-2} + 
		\sum_{\substack{m\geq 1 \\ n\geq 1}} s_{m,n} x^m y^n  \nonumber \\
		&= 2 + \frac{4x}{1-2x} + \frac{C(y) - 1 - y - 2y^2}{y^2} + 
		\sum_{\substack{m\geq 1 \\ n\geq 1}} s_{m,n} x^m y^n \nonumber \\
		&= \frac{4x}{1-2x} + \frac{C(y) - 1 - y}{y^2} + 
		\sum_{\substack{m\geq 1 \\ n\geq 1}} s_{m,n} x^m y^n.\label{eq:s-gf}
	\end{align}
	By Lemma~\ref{lem:s-partition-recursive}, we have the recursive formula
	\[
		s_{m,n} = 2s_{m-1,n} + \sum_{k=1}^n s_{m-1,k-1} C_{n-k+1}
	\]
	for $m,n\geq 1$, which can be rewritten as
	\[
		s_{m,n} = s_{m-1,n} + \sum_{k=0}^n s_{m-1,k} C_{n-k}.
	\]
	We can then use this formula to write
	\begin{align*}
		\sum_{\substack{m\geq 1 \\ n\geq 1}} s_{m,n} x^m y^n &= 
		\sum_{\substack{m\geq 1 \\ n\geq 1}} s_{m-1,n} x^m y^n + 
		\sum_{\substack{m\geq 1 \\ n\geq 1}} \sum_{k=0}^n s_{m-1,k}C_{n-k} x^m y^n \\
		&= x\sum_{\substack{m\geq 0 \\ n\geq 1}} s_{m,n} x^m y^n + 
		x\sum_{\substack{m\geq 0 \\ n \geq 1}} \sum_{k=0}^n s_{m,k} C_{n-k} x^m y^n.
	\end{align*}
	The first term on the righthand side is 
	\begin{align*}
		x\sum_{\substack{m\geq 0 \\ n\geq 1}} s_{m,n} x^m y^n &= 
		x\left( S(x,y) - \sum_{m\geq 0} s_{m,0} x^m \right) \\
		&= x\left( S(x,y) - \frac{2}{1-2x}\right)
	\end{align*}
	and the second term is
	\begin{align*}
		x\sum_{\substack{m\geq 0 \\ n \geq 1}} \sum_{k=0}^n s_{m,k} C_{n-k} x^m y^n &= 
		x\left(\left(\sum_{\substack{m\geq 0 \\ n \geq 0}} s_{m,n} x^m y^n \right)
		\left(\sum_{n\geq 0} C_n y^n \right) - \sum_{m\geq 0} s_{m,0}x^m\right) \\
		&= x\left(S(x,y)C(y) - \frac{2}{1-2x}\right).
	\end{align*}
	Returning to Equation~\ref{eq:s-gf}, we can plug in and cancel terms to obtain 
	\begin{align*}
		S(x,y) &= \frac{C(y) - 1 - y}{y^2} + x S(x,y) + xS(x,y)C(y) \\
		S(x,y)(1-x(1+C(y)) &= \frac{C(y) - 1 - y}{y^2} \\
		S(x,y) &= \left(\frac{C(y)-1-y}{y^2}\right)
		\left(\frac{1}{1-x(1+C(y))}\right),
	\end{align*}
	which completes the proof.
\end{proof}

\bibliographystyle{amsalpha}
\bibliography{nc-cones-semicircles}

\end{document}